\documentclass[leqno]{article}
\usepackage{geometry}
\usepackage{graphicx}
\usepackage[T2A]{fontenc}
\usepackage[utf8]{inputenc}
\usepackage[english]{babel}
\usepackage{hyphenat}
\usepackage{mathtools}
\usepackage{amsfonts,amssymb,mathrsfs,amscd,amsmath,amsthm}
\usepackage{indentfirst}
\usepackage{verbatim}
\usepackage{url}
\usepackage{hyperref}
\usepackage{stmaryrd}
\usepackage{alltt}
\graphicspath{{pictures/}}
\DeclareGraphicsExtensions{.pdf,.png,.jpg}
\usepackage{caption}

\theoremstyle{plain}
\newtheorem{theorem}{Theorem}[section]
\newtheorem{pred}[theorem]{Proposition}

\newtheorem{ass}[theorem]{Assumption}
\newtheorem{lem}[theorem]{Lemma}

\theoremstyle{definition}

\newtheorem{defin}[theorem]{Definition}
\newtheorem{remark}[theorem]{Remark}

\newcommand{\ft}{\operatorname{ft}}

\def\ig#1#2#3#4{\begin{figure}[!ht]\begin{center}%
\includegraphics[height=#2\textheight]{pictures//#1.png}\caption{#4}\label{#3}%
\end{center}\end{figure}}
\title{The Fermat--Torricelli problem in normed spaces}
\author{Daniil A. Ilyukhin}
\date{}

\begin{document}

\maketitle

\begin{abstract}
The article studies a generalization of the classical Fermat--Torricelli problem to normed spaces of arbitrary finite dimension. Necessary and sufficient conditions for the uniqueness of the solution of the Fermat--Torricelli problem for any $n$ points in a fixed space are obtained, and more precise conditions for normed planes and three-dimensional spaces are presented. In addition, examples of the application of the criterion in the norms given by regular polyhedra are given.
\end{abstract}

%%%%%%%%%%%%%%%%%%%%%%%%%%%%%%
\section{Introduction}
%%%%%%%%%%%%%%%%%%%%%%%%%%%%%%

The problem of finding a point that minimizes the sum of distances from it to a given set of points in a metric space was first mentioned in the 17th century. In 1643 Fermat posed a problem for three points on the Euclidean plane, and in the same century Torricelli proposed a solution to this problem (\cite{History}).

\ig{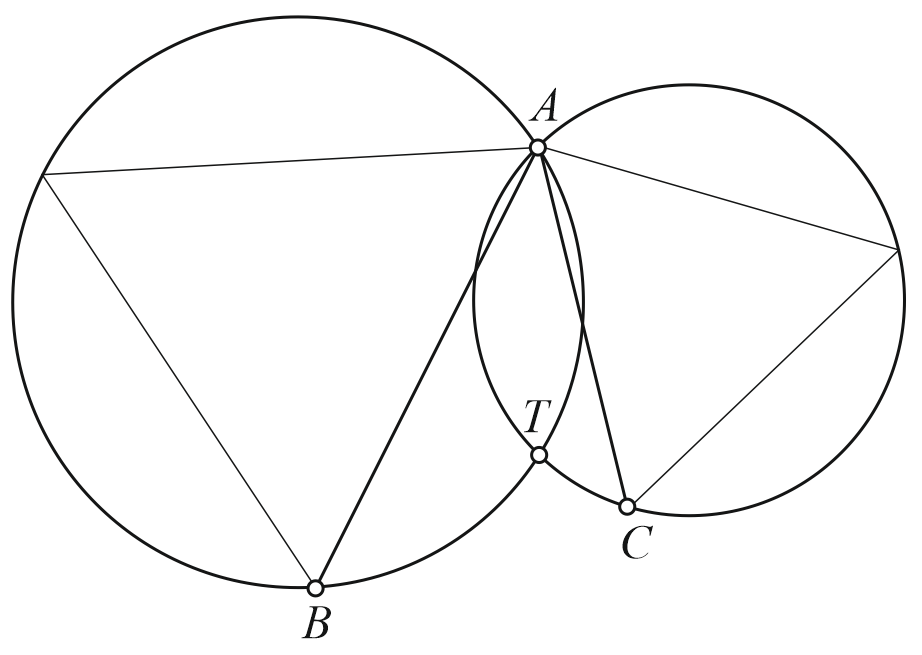}{0.32}{fig:pic4}{The design proposed by Torricelli. The point $T$ is the solution of the problem for the given points $A,B,C$ (\cite{Torricelli})}

Since then, various generalizations of this problem have been considered. The problem was formulated for an arbitrary number of points, the dimension of the space, as well as the norm given in this space. The simplicity of the formulation allows us to consider the problem even in an arbitrary metric space. For example, the problem for four points on the Euclidean plane was solved by D. Fagnano (\cite{bib4}, \cite{bib5}). And for the case of five points, it was proved that the problem is unsolvable in radicals, the proof is given in \cite{bib6} and \cite{bib7}. In addition, there is a generalized problem in which the vertices are considered together with some positive values, called weights. You can read about the development of the weighted problem in the works \cite{bib8}, \cite{arhiv1}, \cite{arhiv2}. In particular, the existence and uniqueness of the solution of such a problem for three points on the Euclidean plane were proved (\cite{bib5}).

This article will consider the classic version of the problem: finding a point for which the minimum sum of distances to elements of a subset of a metric space is reached. We will call such a formulation \emph{generalized Fermat--Torricelli problem} (or simply \emph{Fermat--Torricelli problem}). The work is based on the article \cite{FTproblem}, which describes the application of a geometric approach to the problem and presents some new results that are obtained in the framework of real finite-dimensional normed spaces, called \emph{Minkowski spaces}.

The purpose of this work is to find a necessary and sufficient condition for the uniqueness of the solution of the Fermat--Torricelli problem for any $n$ points in an arbitrary Minkowski space.

I would like to express my gratitude to my scientific adviser, Doctor of Physical and Mathematical Sciences Professor A.A.Tuzhilin and Doctor of Physical and Mathematical Sciences Professor A.O.Ivanov for posing the problem and constant attention to the work.

%%%%%%%%%%%%%%%%%%%%%%%%%%%%%%
\section{The Fermat--Torricelli problem and solving methods}
%%%%%%%%%%%%%%%%%%%%%%%%%%%%%%

All statements in this section will be formulated for the Minkowski space, so this will not be specified.

\begin{defin}
A point $x_0$ is called a \emph{Fermat--Torricelli point} for points $A=\{x_1,\ldots,x_n\}$ if $x=x_0$ minimizes $\sum_{i=1}^n| xx_i|$. The set of all such points will be denoted by $\ft(A)$.
\end{defin}

From the properties of the function $\sum_{i=1}^n|xx_i|$, one can obtain the following assertion about the set of solutions of the Fermat--Torricelli problem, see for example \cite{bib10}.

\begin{pred}
Let $A = \{x_1$,\ldots, $x_n\}$ be points in space. Then $\ft(A)$ is a non-empty, compact and convex set.
\end{pred}

We give two examples of solving the Fermat--Torricelli problem on a plane. The figure \ref{fig:pic5} shows the vertices of an equilateral triangle, first in the Euclidean plane, and then in the norm given by a regular hexagon. In the first case, the set of solutions contains a single point constructed in such a way that the angles between the rays coming out of it in the direction of the vertices of the triangle are equal. In the second case, we specify the location of the points of the given set: let one of them be at the origin, and the other two --- at neighboring vertices of the unit circle. Under such conditions, the set of solutions will include all points of the constructed triangle, including the interior and boundaries.

\ig{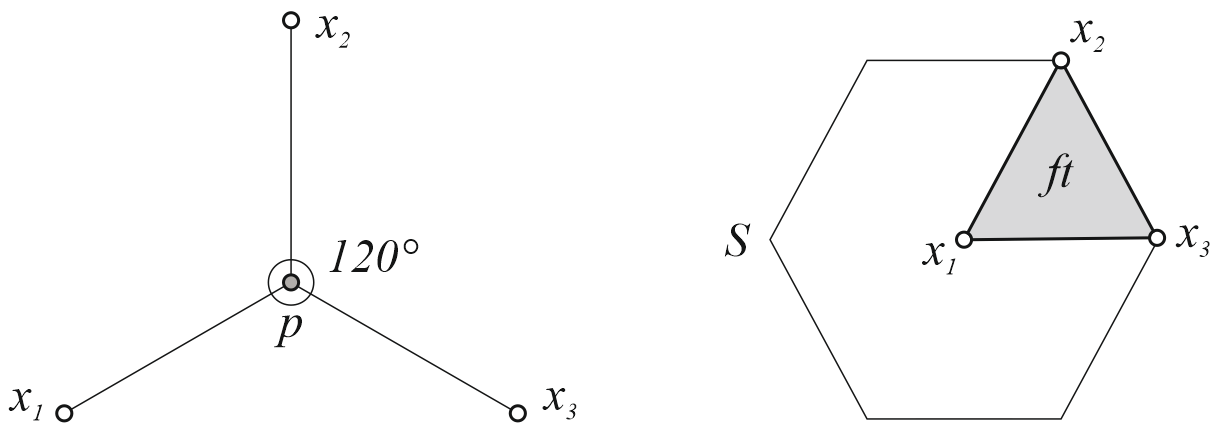}{0.21}{fig:pic5}{Examples of solutions to the Fermat--Torricelli problem on the Euclidean plane and on the $\lambda$-plane}

Now the geometric method for constructing the solution of the Fermat--Torricelli problem will be presented. To use it, along with the original space $X$, one must consider its dual space $X^*$, which consists of linear functionals.

\begin{defin}
A functional $\varphi\in X^*$ is called \emph{norming for a vector} $x\in X$ if $\|\varphi\|=1$ and $\varphi(x)=\|x\|$ .
\end{defin}

It is easy to see that the elements of the space $X$ obtained by multiplying the vector $x$ by a positive number $k$ have the same set of norming functionals. That is, the set of such functionals can be described using points of the unit sphere. Consider norming functionals on some normed plane.

The figure \ref{fig:pic6} shows a section of the unit circle $S$ containing both a flattened and a smooth section. Let us construct norming functionals for vectors starting at zero and ending at a point lying on $S$. At the point $z$ the unit circle has a single support line, then for the corresponding vector there is a unique norming functional $\varphi_4$, its level line coincides with this support line. Internal flattening points $xy$ also correspond to vectors with a unique norming functional. At the point $x$ the unit circle has more than one reference line. Each of them defines a norming functional, that is, for a vector with an end in $x$, there are infinitely many norming functionals. The functionals $\varphi_1$ and $\varphi_2$ are given by the limit positions of the reference line, while $\varphi_3$ is an arbitrary one.

\ig{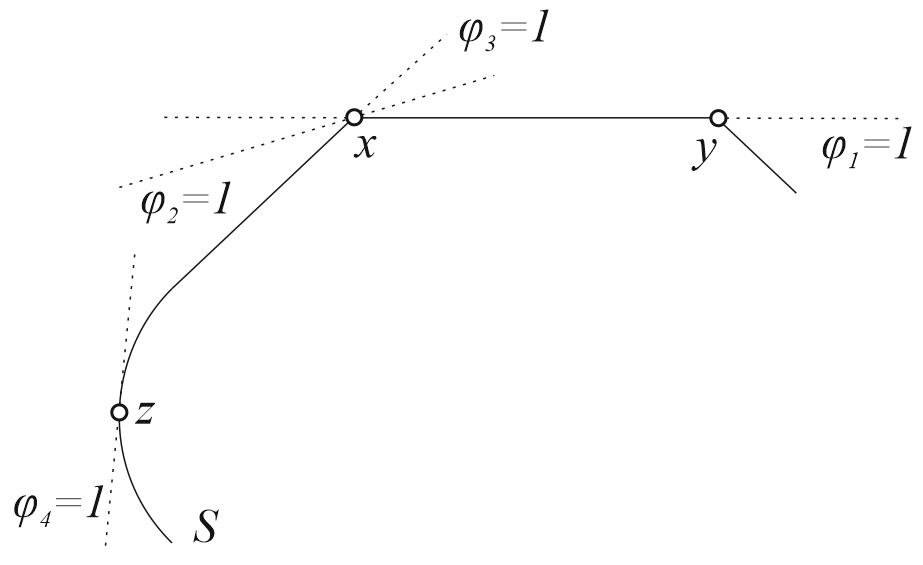}{0.25}{fig:pic6}{Construction of norming functionals along the support lines to the unit circle}

The following theorem is a criterion for a certain point to belong to the set of solutions of the Fermat--Torricelli problem.

\begin{theorem}[\cite{bib11}, \cite{Extreme4}]\label{thm:ftpoint}
Let $x_0$, $x_1$, ..., $x_n$ be points in space and $x_0 \neq x_i$ for $i=1,...,n$. Then $x_0$ is a Fermat--Torricelli point for $A=\{x_1,...,x_n\}$ if and only if each vector $x_i-x_0$, $i=1,.. .,n$, has a norming functional $\varphi_i$ such that $\sum_{i=1}^n\varphi_i=0$.
\end{theorem}

Let us give the simplest example of using this theorem on the Euclidean plane. Let's consider three points on a circle, located at an equal distance from each other, and suppose that the origin of coordinates is the solution of the Fermat--Torricelli problem for them. For each of the points $x_1,x_2,x_3$ there is a unique support line, which is the level line $\varphi_i=1$ of some functional. The three constructed lines form an equilateral triangle and are equidistant from the origin, therefore, the sum of the functionals they define is equal to zero. By the \ref{thm:ftpoint} theorem, the point $p=0$ belongs to the set of solutions.

\ig{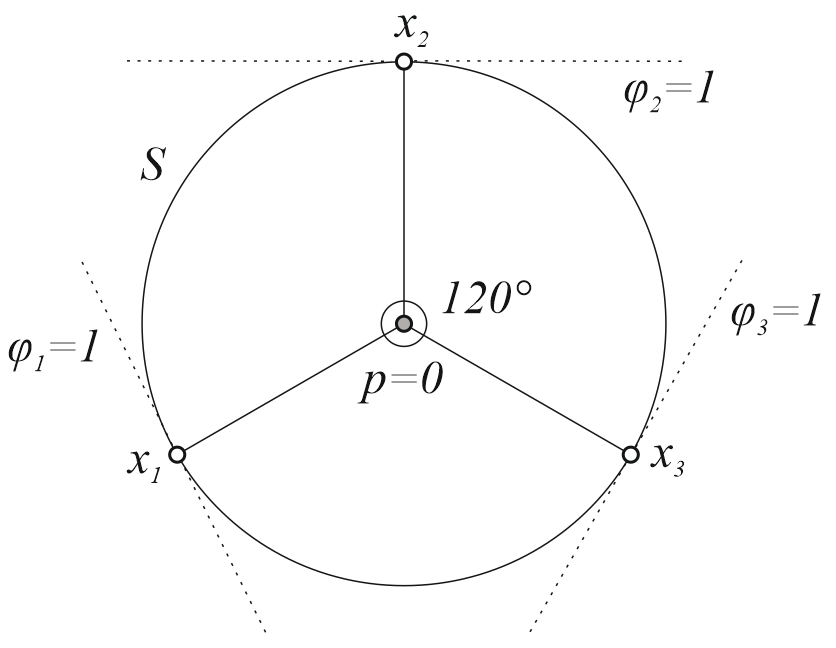}{0.27}{fig:pic8}{The origin of coordinates belongs to the solution set for points $x_1,x_2,x_3$}

Assume that a point is found that is a solution for the set $A$. Now, using the functionals from the \ref{thm:ftpoint} theorem, we can construct the entire set $\ft(A)$. To do this, we introduce a new object.

\begin{defin}
Let a functional $\varphi\in X^*$ and a point $x\in X$ be given. Define \emph{cone} $C(x,\;\varphi)=x-\bigl\{ a:\;\varphi(a)=\|a\|\bigr\}$.
\end{defin}

Let's consider two examples of constructing a cone on the Manhattan plane (figure \ref{fig:pic9}). Let the functional $\varphi_1$ be given by the support line intersecting the unit circle at one point, and $\varphi_2$ be the line containing the flattening. The set $\bigl\{ a:\;\varphi_1(a)=\|a\|\bigr\}$ is a ray that leaves the origin and passes through this point. For the second functional, this will be a whole set of rays intersecting all flattening points. In both cases, we then reflect the constructed set relative to the origin of coordinates and shift it by parallel translation so that the vertex hits the given point.

\ig{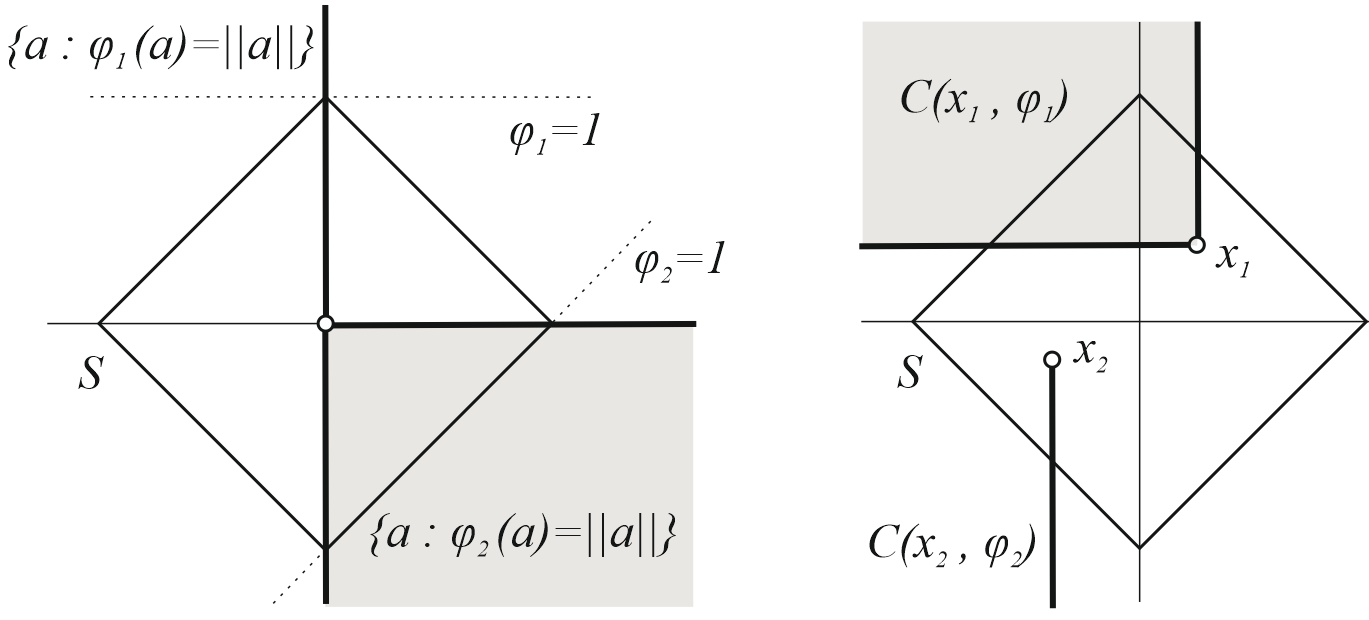}{0.28}{fig:pic9}{Construction of cones on the Manhattan plane}

\begin{theorem}[\cite{bib11}]\label{thm:ftlocus}
Let $A = \{x_1$, ..., $x_n\}$ be points in space and $p\in \ft(A)\setminus A$. By the \ref{thm:ftpoint} theorem, for each vector $x_i-p$, $i=1,...,n$, there exists a norming functional $\varphi_i$ such that $\sum_{i=1}^n \varphi_i=0$. Then $\ft(A)=\cap_{i=1}^n C(x_i,\;\varphi_i)$.
\end{theorem}

The \ref{thm:ftpoint} and \ref{thm:ftlocus} theorems constitute a geometric method for finding a solution to the Fermat--Torricelli problem. To describe the application of this theorem, let's consider in more detail the example with the vertices of an equilateral triangle in the hexagonal norm (figure \ref{fig:pic7}).

First, let's find at least one solution. Take $p=\frac{1}{3}(x_1+x_2+x_3)$ and use the \ref{thm:ftpoint} theorem to prove that $p\in\ft(A)$. To do this, consider the vectors $x_1-p,x_2-p,x_3-p$ and construct their norming functionals. Since they lie in the same directions with internal flattening points, then for each of the vectors $x_i-p$ there is exactly one functional $\varphi_i$ whose level line contains the corresponding flattening. The flattenings are equidistant from the origin and form an equilateral triangle; therefore, the sum of the functionals constructed from them is equal to zero, and the condition of the theorem is satisfied.

\ig{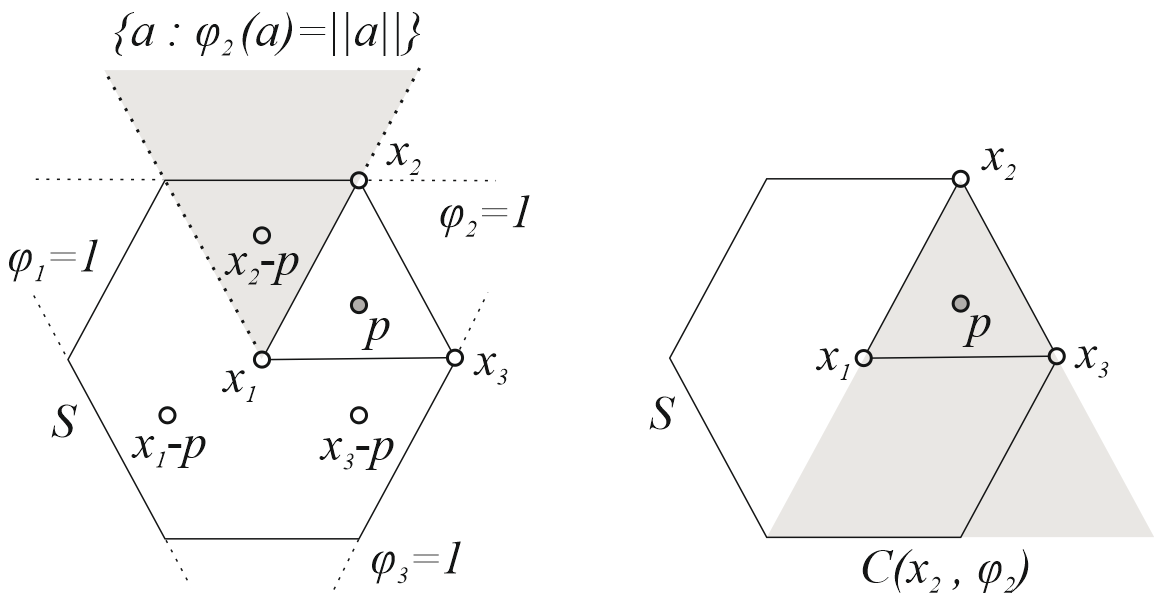}{0.27}{fig:pic7}{Construction of a complete set of solutions for points $x_1,x_2,x_3$ on a plane with hexagonal norm}

Now let's use the \ref{thm:ftlocus} theorem to find all solutions. Let us construct a cone given by the functional $\varphi_2$ and the vector $x_2-p$. Since the functional is norming for all flattening points, the set $\bigl\{ a:\;\varphi_2(a)=\|a\|\bigr\}$ is a set of rays emanating from the origin and passing through the flattening points , that is, an angle whose sides contain two adjacent vertices. Now we will reflect the angle relative to the origin and perform a parallel translation so that the vertex of the angle is at the point $x_2$. After a similar construction of two other cones, we obtain that their intersection is the entire triangle $x_1x_2x_3$.

\begin{remark}
The statement of \ref{thm:ftlocus} theorem does not depend on the choice of the point $p$ and the functionals $\varphi_{i}$.
\end{remark}

The geometric method gives a general description of the solutions of the Fermat--Torricelli problem for various given sets - this is the intersection of some cones with vertices located at the points of this set. In the case of a plane, this observation allows us to formulate the following statements:

\begin{pred}[\cite{FTproblem}]\label{thm:oddset}
Let in space the points of the set $A = \{x_1$, ..., $x_{2k+1}\}$ be located on one straight line in the order of their numbering. Then $\ft(A)=\{x_{k+1}\}$. If $A = \{x_1$, ..., $x_{2k}\}$, then $\ft(A)=\overline{x_kx_{k+1}}$.
\end{pred}

%%%%%%%%%%%%%%%%%%%%%%%%%%%%%%
\section{Uniqueness criterion}
%%%%%%%%%%%%%%%%%%%%%%%%%%%%%%

%%%%%%%%%%%%%%%%%%%%%%%%%%%%%%%%%
\subsection{Criterion for $n$ points in a space of dimension $d$}
%%%%%%%%%%%%%%%%%%%%%%%%%%%%%%%

Let $X$ be a Minkowski space of dimension $d$. In $X$ we pose the Fermat--Torricelli problem for $n$ points. If $x_i\in X, 1\leq i\leq n$, then the solution is the set $\ft(x_1,\ldots,x_n)$, which includes all points $x\in X$ at which the minimum of the function $\sum_{i=1}^n|xx_i|$.

Consider the unit sphere $S$ of the space $X$.

\begin{defin}
\textit{Face} of a unit sphere is its intersection with some supporting hyperplane. If the linear span of the points of a face is a subspace $X$ of dimension $k$, then such a face is called \textit{$(k-1)$-dimensional}.
\end{defin}

\begin{defin}\label{thm:goodset}
Take a finite set of faces of the unit sphere $S$ of the Minkowski space. For each of them, we choose a supporting hyperplane $\pi_i$ that intersects $S$ only along this face. Let the hyperplane $\pi_i$ define the level surface $\varphi_i=1$ of some linear functional $\varphi_i$. If there is a set of supporting hyperplanes such that the sum of the constructed functionals is equal to zero, then we will call such a set of faces \emph{consistent}.
\end{defin}

It follows from the definition that if we take one point from each face of a consistent set, then we obtain a set of points for which $x=0$ is one of the solutions to the Fermat--Torricelli problem. Depending on the dimensions of the faces included in the set and their mutual arrangement, it is possible to obtain various types of solutions, including unique and non-unique ones. It turns out that the problem of finding in the space $X$ a set of points with a non-unique solution is equivalent to the existence in this space of a consistent set of faces with certain properties.

\begin{theorem}\label{thm:criteria}
Let $X$ be a Minkowski space of dimension $d$.

If $n\geq 3$ is odd, then in the space $X$ there are $n$ points for which the solution of the Fermat--Torricelli problem is nonunique if and only if $X$ has a consistent set of $n$ faces of the unit sphere $S$ for which the following condition is satisfied:
\begin{itemize}
  \item The intersection of the linear spans of all faces included in this consistent set contains a line.
\end{itemize}

If $n\geq 4$ is even, then in the space $X$ the solution of the Fermat--Torricelli problem is unique for any $n$ points that do not lie on one straight line if and only if the unit sphere is strictly convex.

\end{theorem}

\begin{proof}
Let $n$ be odd.

Sufficiency. Let us prove that if such a consistent set $A$ exists, then there exists a set of points for which the solution of the Fermat--Torricelli problem is not unique. Take one interior point from each face. Let these be points $x_1, \ldots, x_n$.
By the \ref{thm:ftpoint} theorem, the point $x=0$ is one of the solutions. By the \ref{thm:ftlocus} theorem, the complete solution is the intersection of the cones coming out of the vertices and containing the point $x=0$. Moreover, a $(k+1)$-dimensional cone emerges from the $k$-dimensional face. By assumption, there exists a line $l$ that belongs to the linear span of any face in $A$, that is, to the linear span of any cone. Then each of these cones contains some one-dimensional neighborhood of the point $x=0$ on the line $l$. Therefore, the intersection of all cones also contains a neighborhood of zero. The solution is not unique.

Necessity. Let the space $X$ contain $n$ points $x_1,\ldots,x_n$ for which the solution of the Fermat--Torricelli problem is not unique. Let the point $p$ be included in the solution. Then the solution for the points $x_i-p$ is also not unique, and $x=0$ is one of the solutions. By the \ref{thm:ftlocus} theorem, the set $\ft(x_1-p,\ldots,x_n-p)$ is the intersection of some cones $C_1,\ldots,C_n$ coming out of the points $x_i-p$. Let us consider the norming functionals by which the cones are constructed, namely, the intersections of the unit sphere with the supporting hyperplanes $\varphi_i=1$. We get $n$ faces, which obviously make up a consistent set. Since the solution is non-unique and is an intersection of cones, it contains some non-empty segment containing $x=0$. Each cone also contains this segment. Then the linear span of the cone contains a line containing this segment. Therefore, the intersection of the linear spans of all faces contains a line.

Let $n$ be even.

Let the norm be strictly convex. Then the solution of the Fermat--Torricelli problem for any $n$ points is the intersection of one-dimensional cones with vertices at these points. Since the vertices do not lie on the same line, the cones can only intersect at one point, that is, the solution is always unique.

Let the unit sphere contain a face that is not a point. Take in it any $\frac{n}{2}$ distinct interior points $x_1,\ldots,x_{\frac{n}{2}}$. Consider the solution set for $\pm x_1,\ldots,\pm x_{\frac{n}{2}}$. For each of the points, we define a functional whose level surface is a reference hyperplane defining the face to which the point belongs. Since the faces are opposite and contain the same number of points, the sum of the functionals is equal to zero. By the theorem, \ref{thm:ftpoint} $x=0$ is one of the solutions. The complete solution is the intersection of the cones. All cones have the same dimension, lie in the same subspace of the same dimension, and contain the point $x=0$ as an interior point. Thus, the intersection of all cones contains some neighborhood of zero of the corresponding dimension. The solution is not unique.

The theorem has been proven.
\end{proof}

\begin{remark}
Since the entire space is the linear span of a face of maximum dimension, in the condition of the theorem one can consider the intersection of only non-maximal faces. Thus, if a consistent set consists of only maximal faces, then the additional condition is automatically satisfied.
\end{remark}

If the dimension of the space $d$ is equal to $2$ or $3$, then the equivalent condition can be refined. Let's consider these two cases in more detail.

%%%%%%%%%%%%%%%%%%%%%%%%%%%%%%%%%
\subsection{Criterion for $n$ points on the plane}
%%%%%%%%%%%%%%%%%%%%%%%%%%%%%%%%%

Let $X$ be the Minkowski plane. The Fermat--Torricelli problem for $n$ points is posed in $X$. The unit circle $S$ has only two types of faces. We will call them points and flattenings.

\begin{pred}\label{thm:1cond}
If in $X$ there is a consistent set of faces of the unit circle, consisting of $n$ flattenings, then in this plane there are points $x_1,\ldots,x_n$ for which the set $\ft(x_1,\ldots,x_n)$ - is a non-degenerate polygon.
\end{pred}

\begin{proof}
Let $x_1,\ldots,x_n$ be interior flattening points from a consistent set. Then, taking the functionals $\varphi_1, \ldots, \varphi_n$ from the definition of \ref{thm:goodset}, by the theorem \ref{thm:ftpoint} we get that $p=0$ belongs to $\ft(x_1, \ldots, x_n)$. By the \ref{thm:ftlocus} theorem, $\ft(x_1,\ldots, x_n)$ is the intersection of the cones emerging from the points $x_1, \ldots, x_n$. Since all functionals contain flattenings, all cones are non-degenerate angles. Moreover, each of them contains some full-dimensional neighborhood of zero due to the fact that the points $x_i$ are not boundaries of flattenings. Therefore, the cones at the intersection form a polygon.
\end{proof}

\ig{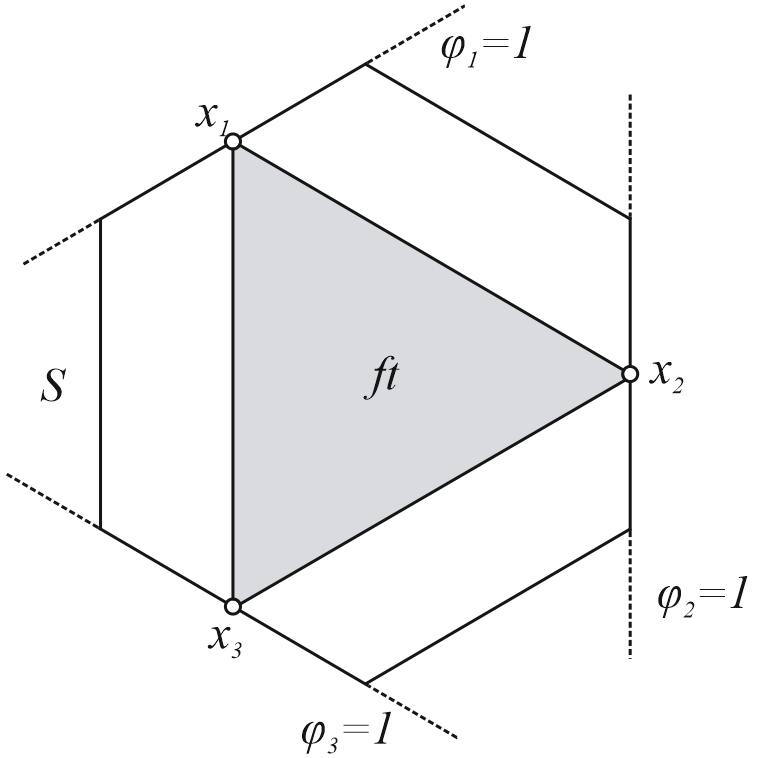}{0.32}{fig:pic2}{The first non-uniqueness condition for $n=3$}

\begin{pred}\label{thm:2cond}
If in $X$ there is a consistent set of faces of the unit circle, consisting of $n-1$ flattening and one point, then in this plane there are $x_1,\ldots,x_n$ for which the set of solutions $\ft(x_1,\ldots ,x_n)$ is a non-degenerate segment.
\end{pred}

\begin{proof}
Let the points $x_1,\ldots,x_{n-1}$ be interior flattening points, and $x_n$ be a point from a consistent set. Similarly to the previous proof, we get that $p=0$ belongs to $\ft(x_1,\ldots,x_n)$. The cones emerging from $x_1,\ldots,x_{n-1}$ contain a neighborhood of zero, and a ray passes from the point $x_n$ and passes through the origin. Consequently, at the intersection we obtain a non-degenerate segment.
\end{proof}

\begin{pred}\label{thm:3cond}
Let $2\leq k\leq n-1$. If there is a consistent set of faces of the unit circle in $X$, consisting of $k$ flattenings and $n-k$ points, and the linear span of the points is a straight line, then there are $x_1,\ldots,x_n$ in this plane for which the set of solutions $ \ft(x_1,\ldots,x_n)$ is a non-degenerate segment.
\end{pred}

\begin{proof}
Let $x_1,\ldots,x_k$ be interior flattening points, and $x_{k+1},\ldots,x_n$ be points from a consistent set. The functionals $\varphi_1,\ldots,\varphi_n$ from the definition of \ref{thm:goodset} satisfy the \ref{thm:ftpoint} theorem, and the point $x=0$ belongs to $\ft(x_1,\ldots,x_n)$ . By the \ref{thm:ftlocus} theorem, the solution $\ft(x_1,x_2,x_3)$ is the intersection of $k$ rays containing the origin and lying on the same line, and $n-k$ non-degenerate angles, each of which contains a neighborhood of zero . We get a segment.
\end{proof}

\ig{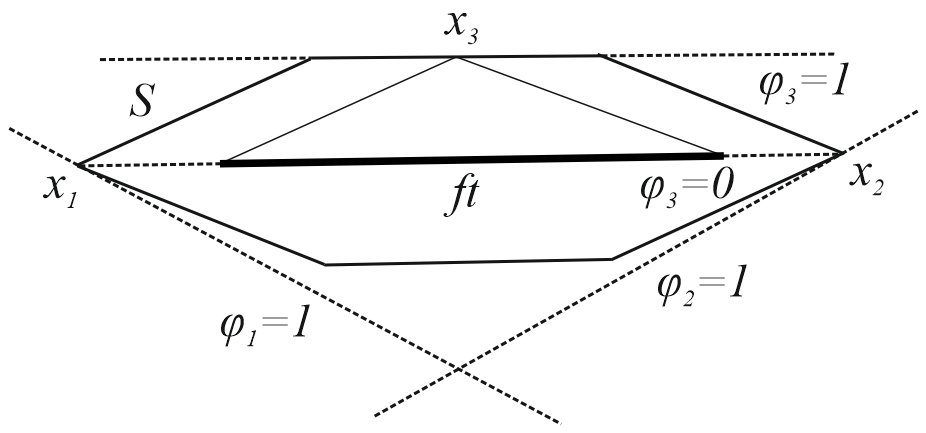}{0.21}{fig:pic3}{The third non-uniqueness condition for $n=3$}

\begin{theorem}\label{thm:criterianx2}
If $n\geq 3$ is odd, then in the normed plane the solution of the Fermat--Torricelli problem is unique for any $n$ points if and only if any consistent set of faces of the unit circle consists of at most $n-2$ flattenings and the dimension of the linear span of face-points is equal to $2$.

If $n\geq 4$ is even, then in the normed plane the solution of the Fermat--Torricelli problem is unique for any $n$ points that do not lie on one straight line if and only if the norm is strictly convex.
\end{theorem}

\begin{proof}
In both statements, the necessity follows from the propositions \ref{thm:1cond}, \ref{thm:2cond}, and \ref{thm:3cond}.
Let us prove sufficiency. Let $n$ be odd and there are points $x_1,\ldots,x_n$ for which the set of solutions of the Fermat--Torricelli problem is not unique. Since the solution is the intersection of cones, it is either a polygon or a segment. In the case of a polygon, all cones are non-degenerate angles. Consider the functionals defining the cones. Their level lines contain flattenings of the unit circle, which obviously constitute a consistent set of faces.

A segment can be obtained by the intersection of at least one non-degenerate angle and rays lying on the same straight line. If all cones are rays, then all points of $x_i$ lie on the same straight line, but by the proposition \ref{thm:oddset} in this case the solution is unique. Similarly, we consider the functionals defining the cones. Since all degenerate cones lie on the same straight line, the points of intersection of the unit circle and the level lines of the functionals also lie on the same straight line. That is, the linear span of the points of the consistent set is equal to $1$.

The statement for even $n$ follows from the \ref{thm:criteria} theorem.

\end{proof}

%%%%%%%%%%%%%%%%%%%%%%%%%%%%%%%%%
\subsection{Criterion for $n$ points in three-dimensional space}
%%%%%%%%%%%%%%%%%%%%%%%%%%%%%%%%%

Let $X$ be a three-dimensional Minkowski space. The Fermat--Torricelli problem for $n$ points is posed in $X$.

\begin{theorem}\label{thm:criterianx3}
If $n\geq 3$ is odd, then there are $n$ points in the three-dimensional Minkowski space $X$ for which the solution of the Fermat--Torricelli problem is nonunique if and only if $X$ contains a consistent set of $n$ faces of the unit sphere $S$, for which the following conditions are satisfied:

\begin{itemize}
  \item The intersection of the linear spans of the face-points included in this consistent set contains a line,
  \item For any pair of face-point and face-segment included in this consistent set, the dimension of their linear span is $2$,
  \item The intersection of the linear shells of the segment faces included in this consistent set contains a straight line.
\end{itemize}

If $n\geq 4$ is even, then in three-dimensional Minkowski space the solution of the Fermat--Torricelli problem is unique for any $n$ points that do not lie on one straight line if and only if the norm is strictly convex.

\end{theorem}

\begin{proof}
Let $n$ be odd.

Sufficiency. Let there be a consistent set of $n$ faces for which all conditions are satisfied. Let us construct a solution of the Fermat--Torricelli problem for a set of interior points of given faces. By the \ref{thm:ftpoint} theorem, the point $x=0$ is included in the solution. The complete solution is the intersection of some cones coming out of the vertices and containing $x=0$ as an interior point.

Let the consistent set have face-points. The linear span of such a face is a straight line containing the point itself and $x=0$. Based on the first condition imposed on a consistent set, this line coincides for all its points. Let this be a straight line $l$. The intersection of all one-dimensional cones coming out of face-points contains a non-empty segment lying in $l$ and containing $x=0$ as an interior point.

Linear spans of faces of the second type are planes containing two-dimensional cones emerging from these elements. By assumption, each such plane contains all point-faces, and hence the line $l$ on which they lie. There is a non-empty segment lying in $l$ and containing $x=0$ as an interior point, which lies in each of the two-dimensional cones. We get that the intersection of all one-dimensional and two-dimensional cones contains at least a segment. If there are no faces-points in a consistent set, then by assumption we take as a line $l$ any line lying at the intersection of the planes.

Since the intersection of three-dimensional cones includes some full-dimensional neighborhood of the point $x=0$, their intersection with the resulting segment is not unique. That is, the intersection of all cones contains at least a segment. Sufficiency has been proven.

Necessity. Let the space $X$ contain $n$ points $x_1,\ldots,x_n$ for which the solution of the Fermat--Torricelli problem is not unique. Let the point $p$ be included in the solution. Then the solution for the points $x_i-p$ is also not unique, and $x=0$ is one of the solutions. By the \ref{thm:ftlocus} theorem, the set $\ft(x_1-p,\ldots,x_n-p)$ is the intersection of some cones $C_1,\ldots,C_n$ coming out of the points $x_i-p$. Let us consider the norming functionals by which the cones are constructed, namely, the intersections of the unit sphere with the support planes $\varphi_i=1$. We obtain a consistent set of faces of the unit sphere. Let us determine what conditions are imposed on this set.

All one-dimensional cones coming out of point-faces must lie on the same straight line, otherwise their intersection will contain only the point $x=0$. In other words, the intersection of the linear spans of these faces is this straight line. Let this be a straight line $l$.

Two-dimensional cones emerge from the faces-segments. Since the solution is not unique, let their intersection contain a non-empty segment containing $x=0$. Let also this segment belong to the line $l$ if the consistent set contains points and such a line is constructed. Then the linear span of any face-segment, that is, the plane containing the corresponding cone, contains the line containing this segment. We have found that the intersection of the linear spans of the segment faces contains a straight line. In the presence of face-points, this will be a straight line $l$. Hence, a plane containing a face-segment contains any face-point. That is, the linear span of any pair of face-point and face-segment has dimension $2$. All conditions on the consistent set are satisfied.

The statement for even $n$ follows from the \ref{thm:criteria} theorem.

\end{proof}

%%%%%%%%%%%%%%%%%%%%%%%%%%%%%%
\section{Regular polyhedra and three-dimensional normed spaces}
%%%%%%%%%%%%%%%%%%%%%%%%%%%%%%

The article \cite{mywork} showed the application of the criterion on lambda planes --- normed planes defined by regular polygons. Consider the problem in some three-dimensional spaces.

\begin{ass}
There are only five regular polyhedra: tetrahedron, cube, octahedron, icosahedron and dodecahedron. Four of them, all except the tetrahedron, set the norm in three-dimensional space.
\end{ass}

Let's check whether the solution of the Fermat--Torricelli problem is unique for any three points in the given spaces. To do this, it is necessary to study the consistent sets of faces existing in them.

\begin{lem}\label{thm:lemma1}
If three faces of the unit sphere of the Minkowski space constitute a consistent set, then they are pairwise disjoint.
\end{lem}

\begin{proof}
Each face from the consistent set corresponds to a functional whose level surface contains the given face. The sum of the functionals of all faces is equal to zero.
If some two faces have a common point, then the sum of their functionals at this point is equal to $2$. Then the third functional takes the value $-2$ at this point, and it is not norming. Contradiction.
\end{proof}

\begin{pred}
In a normed space given by a regular cube, there are three points for which the solution of the Fermat--Torricelli problem is not unique. Moreover, if the solution is not unique, then it is a segment.
\end{pred}

\begin{proof}
Let $X$ be the space under consideration. Consider in $X$ all possible variants of consistent sets.

Let there be a consistent set in $X$, and let at least two of them be two-dimensional faces. By the \ref{thm:lemma1} lemma, they have no common points. Hence, they are opposite, and the sum of their functionals is equal to zero. But then the third functional is zero. Contradiction.

Let a consistent set contain exactly one two-dimensional face. It cannot share points with either of the other two faces. If the remaining two faces are cube edges, then they belong to the opposite two-dimensional face and are not adjacent. The functional corresponding to a two-dimensional face takes the value $-1$ on both edges. Then each of the other two functionals is equal to zero on the other edge. This means that their level planes make an angle of $45$ degrees with the neighboring faces of the cube. The planes intersect in a straight line on which the sum of two functionals is equal to $2$. However, the first functional takes the value $-\frac{3}{2}$ on this line. The sum of functionals is not equal to zero. Contradiction. The second and third faces also cannot be an edge and a point or two points, since in this case their linear span does not contain a line, as it should be in the criterion condition.

Let the matched set consist of three edges. If some two edges belong to the same face, then the third necessarily belongs to the opposite one. Let it be symmetrical to one of the first two. Then the problem is reduced to constructing functionals for three vertices of the square that defines the norm on the plane. Such functionalities exist. Let the level lines of the functionals corresponding to the symmetric vertices intersect on the extension of the diagonal of the square at the point of the norm $2$. And the level line of the third functional passes through the remaining vertex perpendicular to this diagonal. Then the sum of the functionals is equal to zero. The functionals for the cube edges are constructed similarly. The intersection of the linear shells of the edges is a straight line. Taking three arbitrary interior points on these edges, in the solution we obtain a non-empty segment lying on this line.

If the third edge is non-parallel, then the intersection of the line spans of all edges is a point. If no pair of edges belongs to the same face, then again we get a point at the intersection of their linear envelopes. The criterion condition is not fulfilled.

If a consistent set consists of two edges and a point, then this point must lie in the same two-dimensional subspace with each of the edges. In this case, we get that the point is the end of one of the two edges. If a consistent set consists of one edge and two vertices, then the line passing through the vertices is the diagonal of the cube, and it must lie in the same plane as the edge, but there is no such edge.

Thus, in $X$ there are only consistent sets consisting of three parallel edges, and the solution of the Fermat--Torricelli problem is either a segment or a point. The statement is proven.

\end{proof}

\begin{pred}
In a normed space given by a regular octahedron, the solution of the Fermat--Torricelli problem is unique for any three points.
\end{pred}

\begin{proof}
In a normed space, in which the unit sphere is a regular octahedron, the norm for a point with coordinates $a = (x,y,z)$ is given by the formula $\|a\|=|x|+|y|+|z|$. Then for the points $a_i = (x_i,y_i,z_i)$ and the point $a = (x,y,z)$ we have
$$\sum_{i=1}^n|aa_i|=\sum_{i=1}^n\|a-a_i\|=\sum_{i=1}^n(|x-x_i|+|y -y_i|+|z-z_i|)=\sum_{i=1}^n|x-x_i|+\sum_{i=1}^n|y-y_i|+\sum_{i=1}^n |z-z_i|$$
That is, the problem is reduced to one-dimensional, and according to \ref{thm:oddset} the solution for three points located on the same straight line is unique. The assertion has been proven.
\end{proof}

\begin{pred}
In a normed space given by a regular dodecahedron, there are three points for which the solution of the Fermat--Torricelli problem is not unique. Moreover, if the solution is not unique, then it is a segment.
\end{pred}

\begin{proof}
Consider in $X$ all possible variants of consistent sets.
For convenience, we introduce notation and formulations for the elements of the dodecahedron. $f_i$ are used to designate two-dimensional faces. The edge belonging to the faces $f_i, f_j$ will be denoted by $e_{i,j}$, and the vertex belonging to the faces $f_i, f_j, f_k$ will be denoted by $v_{i,j,k}$. The faces $f_1$ and $f_{12}$ are called top and bottom respectively. Edges and vertices belong to the equator of the dodecahedron if they do not intersect with the top and bottom faces.

\ig{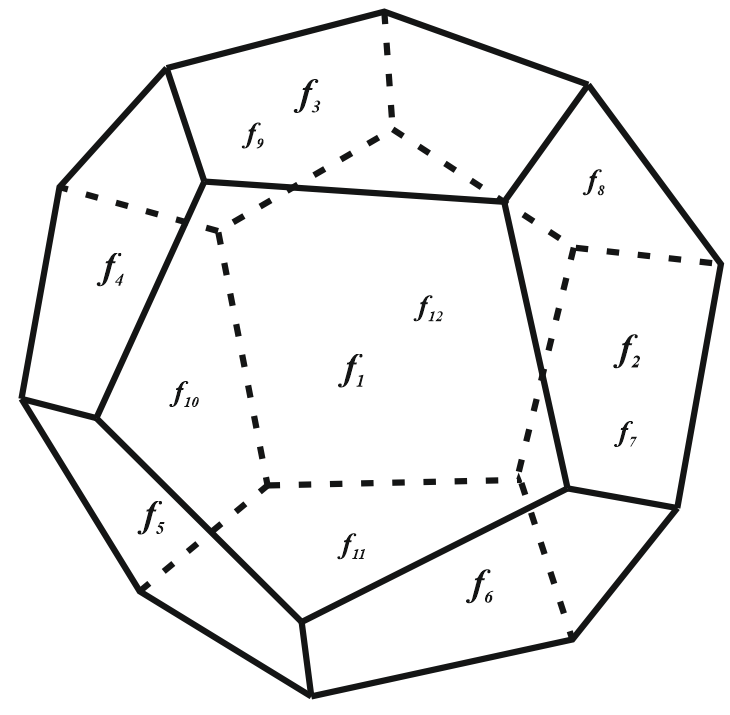}{0.35}{fig:pic10}{Regular dodecahedron sets the norm in three-dimensional space}

Assume that a consistent set contains at least two two-dimensional faces. By the \ref{thm:lemma1} lemma, they cannot be adjacent and opposite, so we will assume that these are the faces $f_1$ and $f_7$. If the third element of the consistent set is also a two-dimensional face, then it can only be $f_9$ or $f_{10}$. Since they are located symmetrically with respect to the first two faces, they simultaneously complement or do not complement them to a consistent set. But if they complement, then their functionals must match, and this is not true.

Let the third element be an edge. All edges except $e_{9,10}$ are symmetrical with respect to the two considered faces, so they cannot be included in the consistent set. The functional of the edge $e_{9,10}$ is equal to $-1$ at the vertex $v_{2,6,7}$. At the same vertex, the functional of $f_7$ is equal to $1$, but the functional of $f_1$ is not equal to zero. The sum of functionals is not equal to zero. 

Let the third element be a vertex. Similar to the previous reasoning, it suffices to check only the vertex $v_{4,9,10}$. In this case, at the vertex $v_{2,6,7}$ the sum of the functionals is again not equal to zero.

Consider the variant when the consistent set contains a single face $f_1$. Let the other two elements be edges. If one of the edges belongs to the bottom face, then the zero level of the functional of the second edge passes through it. But in this case it is one of the edges adjoining the top face, which is impossible by the \ref{thm:lemma1} lemma.
Let one of the edges connect the equator and the bottom face. But then any of the remaining edges has a symmetric relative to the considered face and edge and cannot complement them to a consistent set.
Then both edges must belong to the equator, and they are symmetrical about the center. The plane passing through these edges is the zero level of the $f_1$ face functional, but is not parallel to it, which is impossible.

Let the face $f_1$ be complemented by an edge and a vertex lying in the same plane. The subspaces defined by edges from the equator contain only vertices that are the ends of the edge itself and the opposite edge. Each of them is symmetrical to itself. Similarly, the edge of the bottom face also cannot be included in the consistent set. There are edges between the bottom face and the equator. It suffices to check the edge $e_{8,9}$ and the vertex $v_{5,6,11}$. In this case, on the line connecting the vertices $v_{5,6,11}$ and $v_{3,8,9}$ the face functional $f_1$ is equal to zero. But this line is not parallel to the plane of the face. Contradiction.

If the face $f_1$ is complemented by two symmetrical vertices, then this is necessarily a pair of vertices from the equator. But the straight line passing through them is not parallel to the plane of the face. A contradiction, since the face functional on this line must be equal to zero.

There are variants of consistent sets that do not contain a single face. Let us assume that a consistent set consists of three edges whose intersection of subspaces contains at least a line. Let two of them be opposite, for example $e_{1,4}$ and $e_{7,12}$. Their common two-dimensional subspace defines the zero level of the third edge functional. Only the edges $e_{2,6}$ and $e_{9,10}$ are parallel to this plane. They are symmetrical, so it suffices to consider the first.

\ig{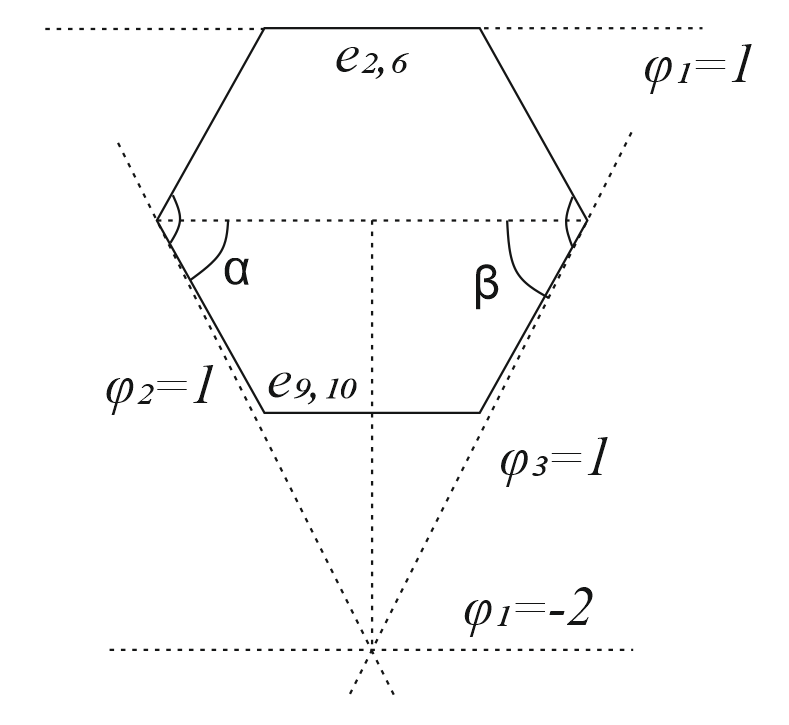}{0.35}{fig:pic11}{}
Let the level of the functional $\varphi_1$ pass through the edge $e_{2,6}$ and make equal dihedral angles with neighboring faces. The functional levels $\varphi_2$ and $\varphi_3$ contain the edges $e_{1,4}$ and $e_{7,12}$ respectively. In the plane of these two edges, the sum of the functionals is equal to zero. Now consider a plane containing the edge $e_{2,6}$ and the midpoints of the edges $e_{1,4}$ and $e_{7,12}$. The section of the dodecahedron by this plane is shown in \ref{fig:pic11}. Since the sum of the functionals is equal to zero, the point of intersection of the level lines $\varphi_2=1$ and $\varphi_3=1$ belongs to the line $\varphi_1=-2$. Let $\alpha$ be the dihedral angle between the level $\varphi_1=0$ and the face, and $\beta$ be the angle between the levels $\varphi_1=0$ and $\varphi_2=1$. If $\beta >\alpha$, then such functionals exist, and so does a consistent set of corresponding edges. We will prove that this is indeed the case. The angle $\alpha$ is half the dihedral angle of a regular dodecahedron, i.e. $$\alpha=\frac{1}{2}\arccos{(-\frac{1}{\sqrt{5}})}$$
The lengths of the edges $e_{2,6}$ and $e_{9,10}$ are equal to $a=1$. The lengths of the remaining sides of the resulting hexagon in the section are equal to the height of a regular pentagon, that is, $b=\frac{\sqrt{5+2\sqrt{5}}}{2}$. Calculate the legs of a right triangle with hypotenuse $b$ and acute angle $\alpha$:
$$c=b\cos{\alpha}=\frac{1}{2}\sqrt{\frac{3+\sqrt{5}}{2}}, d=b\sin{\alpha}=\frac{1}{2}\sqrt{\frac{7+3\sqrt{5}}{2}}$$
Now we can calculate the tangent of the angle $\beta$:
$$\tan{\beta}=\frac{2d}{c+\frac{a}{2}}=\frac{\sqrt{\frac{7+3\sqrt{5}}{2}}}{ \frac{1}{2}\sqrt{\frac{3+\sqrt{5}}{2}}+\frac{1}{2}}=2$$
Since $\arctan{2}>\frac{1}{2}\arccos{(-\frac{1}{\sqrt{5}})}$, then $\beta>\alpha$, and the edges of the dodecahedron form a consistent set. Since the two-dimensional subspaces defined by these three edges intersect along a straight line, this set gives non-unique solutions to the Fermat--Torricelli problem, moreover, belonging to some straight line, that is, segments.

If among the three edges of the dodecahedron there is not a single pair of opposite ones, then the two-dimensional subspaces containing the edges intersect no more than at a point.

A consistent dodecahedron set cannot consist of two edges and a vertex, so in this case the vertex must lie in the same plane with each edge. We get that the edges are opposite, and the vertex is the end of one of the edges, which is impossible by the lemma \ref{thm:lemma1}.

Also, there is no consistent set of one edge and two vertices lying in the same plane. If so, then by the lemma \ref{thm:lemma1}, the vertices are the ends of the edge opposite the given one. Using the picture \ref{fig:pic11}, let's say that the edge $e_{2,6}$ is being considered, and the vertices are the endpoints of the edge $e_{9,10}$. Then the level lines of the functionals $\varphi_2$ and $\varphi_3$ pass through these vertices and intersect at the lines $\varphi_1=-2$. It can be seen from the figure that this is not possible.

Thus, the dodecahedron has only consistent sets consisting of three edges. Moreover, if a non-unique solution is obtained, then this is a segment. The assertion has been proven.

\end{proof}

\begin{pred}
If the unit sphere of a three-dimensional normed space is a prism, then for any $n\geq 2$ there are $n$ points for which the solution of the Fermat--Torricelli problem is not unique.
\end{pred}

\begin{proof}
Consider a normed plane whose unit circle coincides with the base of the prism. In this normed plane, we take any consistent set and the functionals defining it. Let us draw the level lines of these functionals at the two bases of the prism. A unique set of support planes passes through them, defining $n$ functionals in the original space. Their sum is equal to zero, and the linear span of the faces contains a line. The condition of the criterion \ref{thm:criterianx3} is satisfied, and the assertion is proved.
\end{proof}

\label{end}

\end{document}